\documentclass[letterpaper,reqno,12pt]{amsart}
\usepackage[margin=1.2in]{geometry}

\usepackage{amsmath} \usepackage{amssymb} \usepackage{amsthm}
\usepackage{amscd} 
\usepackage[all]{xy} 
\usepackage{enumerate}
\usepackage{mathrsfs}
\usepackage[dvipdfm]{hyperref}
\hypersetup{colorlinks=true}

\def\ge{\geqslant}
\def\le{\leqslant}
\def\phi{\varphi}
\def\epsilon{\varepsilon}
\def\tilde{\widetilde}

\def\to{\longrightarrow}
\def\mapsto{\longmapsto}

\def\Hom{\operatorname{Hom}}
\def\Spec{\operatorname{Spec}}

\def\Div{\operatorname{div}}

\newcommand{\N}{\mathbb{N}}
\newcommand{\Q}{\mathbb{Q}} 
 
\newcommand{\R}{\mathbb{R}}

\newcommand{\sO}{\mathcal{O}}
\newcommand{\fa}{\mathfrak{a}}
\newcommand{\fb}{\mathfrak{b}}

\newcommand{\m}{\mathfrak{m}}

\newcommand{\U}{\mathfrak{U}}

\newcommand{\Tr}{\mathrm{Tr}}

\newcommand{\emb}{\mathrm{emb}}
\newcommand{\sh}{\mathrm{sh}}

\newcommand{\FPT}{\mathrm{FPT}}
\newcommand{\fpt}{\mathrm{fpt}}
\newcommand{\lfpt}{\mathrm{fpt}}

\newcommand{\ulim}{\operatorname{ulim}}

\newcommand{\NN}{\mathcal{N}}
\newcommand{\D}{\Delta}
\newcommand{\age}[1]{\lceil #1 \rceil}
\newcommand{\qadic}[3][q]{\langle #2 \rangle_{#3}}

\newcommand{\ultra}[1]{{}^* #1}
\newcommand{\cata}[1]{#1_\#}
\newcommand{\catae}[1]{[ #1 ]_m}
\newcommand{\ntau}[3][X]{\tau(\omega_{#1}, {#2}_{- 0}, #3)}
\newcommand{\ntri}[3][X]{(\omega_{#1}, {#2}_{- 0}; #3)}

\theoremstyle{plain}
\newtheorem{thm}{Theorem}[section] 
\newtheorem{cor}[thm]{Corollary}
\newtheorem{prop}[thm]{Proposition}

\newtheorem*{mainthm}{Main Theorem}

\newtheorem{lem}[thm]{Lemma}
\theoremstyle{definition} 
\newtheorem{defn}[thm]{Definition}
\newtheorem{propdef}[thm]{Proposition-Definition}

\theoremstyle{remark}
\newtheorem{rem}[thm]{Remark}

\newtheorem*{acknowledgement}{Acknowledgments}

\title{Ascending chain condition for $F$-pure thresholds with fixed embedding dimension}
\author{Kenta Sato}
\address{Graduate School of Mathematical Sciences, University of Tokyo, 3-8-1 Komaba, Meguro-ku, Tokyo 153-8914, Japan}
\email{ktsato@ms.u-tokyo.ac.jp}


\begin{document}
\tolerance = 9999

\maketitle
\markboth{KENTA SATO}{ACC FOR $F$-PURE THRESHOLDS WITH FIXED EMBEDDING DIMENSION}

\begin{abstract}
In this paper, we prove that the set of all $F$-pure thresholds of ideals with fixed embedding dimension satisfies the ascending chain condition.
As a corollary, given an integer $d$, we verify the ascending chain condition for the set of all $F$-pure thresholds on all $d$-dimensional normal l.c.i. varieties.
In the process of proving these results, we also show the rationality of $F$-pure thresholds of ideals on non-strongly $F$-regular pairs.
\end{abstract}

\section{Introduction}
A ring $R$ of characteristic $p>0$ is said to be \emph{$F$-finite} if the Frobenius morphism $F: R \to R$ is finite.
Suppose that $X$ is a normal variety over an $F$-finite field $k$ of characteristic $p>0$.
We further assume that $X$ is sharply $F$-pure, that is, the Frobenius homomorphism $\sO_X \to F_*\sO_X$ locally splits.
Then, for every coherent ideal sheaf $\fa \subsetneq \sO_X$, we can define the \emph{$F$-pure threshold} $\fpt(X; \fa) \in \R_{\ge 0}$ in terms of Frobenius splittings (see Definition \ref{fpt def} below).
Recent studies (\cite{TW}, \cite{Taka}, \cite{HnBWZ}) reveal that $F$-pure thresholds have a strong connection to log canonical thresholds in characteristic $0$. 
Moreover, as seen in \cite{TW}, \cite{MTW} and \cite{BS}, the $F$-pure threshold itself is an interesting invariant in both commutative algebra and algebraic geometry in positive characteristic.

In \cite{Sat}, motivated by the ascending chain condition for log canonical thresholds in characteristic $0$ (\cite{Sho}, \cite{dFEM}, \cite{dFEM2} and \cite{HMX}), the author studied the ascending chain condition for $F$-pure thresholds.
In {\it loc. cit.}, it was proved that the set of all $F$-pure thresholds of ideals on a fixed germ of a strongly $F$-regular pair satisfies the ascending chain condition, where strong $F$-regularity is a stronger condition than sharp $F$-purity (see Definition \ref{Fsing def}).

In this paper, we extend the result of \cite{Sat} to the case of sharply $F$-pure pair under some conditions.
The first result in this paper deals with the ascending chain condition for $F$-pure thresholds on l.c.i. varieties, which is an positive characteristic analogue of \cite[Theorem 1.3]{dFEM}.

\begin{thm}[Corollary \ref{lci ACC}]\label{intro lci}
Fix an integer $n \ge 1$ and an $F$-finite field $k$ of characteristic $p>0$.
Let $T$ be a set of all $n$-dimensional normal l.c.i. varieties over $k$ which are sharply $F$-pure.
Then the set
\[
\left\{ \fpt(X; \fa) \mid X \in T , \fa \subsetneq \sO_X \right\}
\]
satisfies the ascending chain condition.
\end{thm}

In the proof of \cite[Theorem 1.3]{dFEM}, they use Inversion of Adjunction, which is the reason why they consider l.c.i. varieties.
On the other hand, the theory of $F$-adjunction, introduced by Schwede (\cite{Sch-adj}), can be applied even if the variety is not a locally complete intersection.
Therefore, we can employ the same strategy in a more general setting as that of \cite{dFEM}.

Suppose that $(R,\m)$ is an $F$-finite Noetherian normal local ring of characteristic $p>0$ and $\D$ is an effective $\Q$-Weil divisor on $\Spec R$.
We further assume that the pair $(R,\D)$ is sharply $F$-pure (see Definition \ref{Fsing def}).
In this case, we can define the $F$-pure threshold $\fpt(R, \D; \fa) \in \R_{\ge 0}$ for every proper ideal $\fa \subseteq R$.
The following is the main theorem of this paper, which extends the main theorem of \cite{Sat}.

\begin{mainthm}[Theorem \ref{main}]\label{intro main}
Fix positive integers $e$ and $N$.
Suppose that $T$ is any set such that every element of $T$ is an $F$-finite Noetherian normal local ring $(R,\m, k)$ with $\dim_k (\m/\m^2) \le N$.
Let $\FPT(T,e) \subseteq \R_{\ge 0}$ be the set of all $F$-pure thresholds $\fpt(R, \D ;\fa)$ such that
\begin{itemize}
\item $R$ is an element of $T$,
\item $\fa$ is a proper ideal of $R$, and
\item $\D$ is an effective $\Q$-Weil divisor on $X=\Spec R$ such that $(R,\D)$ is sharply $F$-pure and $(p^e-1)(K_X+\D)$ is Cartier, where $K_X$ is a canonical divisor on $X$.
\end{itemize}
Then the set $\FPT(T, e)$ satisfies the ascending chain condition.
\end{mainthm}

In the process of proving the main theorem, we treat the rationality problem for $F$-pure thresholds.
In characteristic $0$, since log canonical thresholds can be computed by a single log resolution, it is obvious that the log canonical threshold of any ideal on any log $\Q$-Gorenstein pair is a rational number.
In \cite{dFEM}, they use the rationality to reduce the ascending chain condition for log canonical thresholds on l.c.i. varieties to that on smooth varieties.

However, in positive characteristic, the rationality of $F$-pure thresholds is a more subtle problem.
In \cite{ST14}, Schwede and Tucker proved that the $F$-pure threshold of any ideal on any log $\Q$-Gorenstein strongly $F$-regular pair is a rational number.
In this paper, we generalize their result to the case where the pair is not necessarily strongly $F$-regular, under the assumption that the Gorenstein index is not divisible by the characteristic.

\begin{thm}[Corollary \ref{rat}]\label{intro rat}
Suppose that $(R, \m)$ is an $F$-finite Noetherian normal local ring of characteristic $p>0$ and $\D$ is an effective $\Q$-Weil divisor on $X=\Spec R$ such that $(R,\D)$ is sharply $F$-pure and $K_X+\D$ is $\Q$-Cartier with index not divisible by $p$.
Then the $F$-pure threshold $\fpt(R,\D; \fa)$ is a rational number for every proper ideal $\fa \subseteq R$.
\end{thm}

In the proof of Theorem \ref{intro rat}, we introduce a new variant of parameter test modules.
Assume that $\D$ is $\Q$-Cartier and $t \ge 0$ is a real number.
Then, we define the submodule $\ntau{\D}{\fa^t} \subseteq \omega_X$ as an approximation of the parameter test module $\tau(\omega_X, \D, \fa^t) \subseteq \omega_X$ by small perturbations of $\D$ (see Definition \ref{ntau def}).
A real number $t \ge 0$ is called an \emph{$F$-jumping number} of $\ntri{\D}{\fa}$ if one of the following hold:
\begin{enumerate}
\item for every $\epsilon>0$, we have $\ntau{\D}{ \fa^t} \subsetneq \ntau{\D}{\fa^{t-\epsilon}}$, or
\item for every $\epsilon>0$, we have $\ntau{\D}{\fa^t} \supsetneq \ntau{\D}{\fa^{t+\epsilon}}$.
\end{enumerate}
The key ingredient of the proof of Theorem \ref{intro rat} is the rationality of $F$-jumping numbers of $\ntri{\D}{\fa}$ (Corollary \ref{F-jump rat}).
%
Theorem \ref{intro rat} follows from the rationality and $F$-adjunction because the $F$-pure threshold $\fpt(R, \D; \fa)$ is equal to the first $F$-jumping number if $X=\Spec R$ is regular (Proposition \ref{fpt fjn}). 
We use Theorem \ref{intro rat} to reduce the main theorem to the ascending chain condition for $F$-pure thresholds of ideals on a fixed $F$-finite regular local ring, which has already been proved in \cite{Sat}.

\begin{acknowledgement}
The author wishes to express his gratitude to his supervisor Professor Shunsuke Takagi for his encouragement, valuable advice and suggestions.
He was supported by JSPS KAKENHI Grant Number 17J04317.
\end{acknowledgement}
\section{Preliminaries}

\subsection{$F$-singularities}

In this subsection, we recall the definitions and some basic properties of $F$-singularities.

A ring $R$ of characteristic $p>0$ is said to be \emph{$F$-finite} if the Frobenius morphism $F: R \to R$ is a finite ring homomorphism.
A scheme $X$ is said to be \emph{$F$-finite} if for every open affine subscheme $U \subseteq X$, $\sO_U$ is $F$-finite.
If $R$ is an $F$-finite Noetherian normal domain, then $R$ is excellent (\cite{Kun}) and $X=\Spec R$ has a \emph{dualizing complex} $\omega_X^{\bullet}$, a \emph{canonical module} $\omega_X$ and a \emph{canonical divisor} $K_X$ (see for example \cite[p.4]{ST17}).

Through this paper, all rings will be assumed to be $F$-finite of characteristic $p>0$.

\begin{defn}
A \emph{pair} $(R, \D)$ consists of an $F$-finite Noetherian normal local ring $(R,\m)$ and an effective $\Q$-Weil divisor $\D$ on $X$.
A \emph{triple} $(R,\D, \fa^t)$ consists of a pair $(R,\D)$ and a symbol $\fa^t$, where $\fa \subseteq R$ is an ideal and $t \ge 0$ is a real number.
\end{defn} 

\begin{defn}\label{Fsing def}
Let $(R,\D, \fa^t)$ be a triple.
\begin{enumerate}
\item
$(R, \D, \fa^t)$ is said to be \emph{sharply $F$-pure} if there exist an integer $e>0$ and a morphism $\phi \in \Hom_R(F^e_* R( \age{(p^e-1) \Delta}), R)$ such that $\phi( F^e_* \fa^{\age{t (p^e -1)}})=R$.
\item
$(R, \D, \fa^t)$ is said to be \emph{strongly $F$-regular} if for every non-zero element $c \in R$, there exist an integer $e>0$ and a morphism $\phi \in \Hom_R(F^e_* R( \age{(p^e-1) \Delta}), R)$ such that $\phi( F^e_* (c \fa^{\age{t (p^e -1)}}))=R$.
\end{enumerate}
\end{defn}

\begin{lem}\label{F-sing basic}
Let $(R, \D, \fa^t)$ be a triple.
Then the following hold.
\begin{enumerate}
\item If $(R, \D, \fa^t)$ is strongly $F$-regular, then it is sharply $F$-pure.
\item Suppose that $0 \le \D' \le \D$ is an $\Q$-Weil divisor and $0 \le t' \le t$ is a real number.
If $(R, \D, \fa^t)$ is strongly $F$-regular (resp. sharply $F$-pure), then so is $(R, \D', \fa^{t'})$.
\item If $(R, \D)$ is strongly $F$-regular and $(R, \D, \fa^t)$ is sharply $F$-pure, then $(R, \D, \fa^{s})$ is strongly $F$-regular for every $0 \le s<t$.
\item Let $f \in R$ be a non-zero element and $\fb : = f \cdot \fa \subseteq R$.
Then, $(R, \D + t \Div_R (f), \fa^t)$ is strongly $F$-regular if and only if $(R, \D, \fb^t)$ is strongly $F$-regular.
\item Let $\widehat{R}$ be the $\m$-adic completion and $\widehat{\D}$ the flat pullback of $\D$ to $\Spec \widehat{R}$.
Then, $(R,\D, \fa^t)$ is sharply $F$-pure if and only if $(\widehat{R}, \widehat{\D}, (\fa \widehat{R})^t)$ is sharply $F$-pure.
\end{enumerate}
\end{lem}

\begin{proof}
(1) and (2) follow from definitions.
The proof of (3) is similar to that of \cite[Proposition 2.2 (5)]{TW}.
(4) follows from \cite[Lemma 3.1]{Sch-note}.
For (5), we define $I : = \sum_{e, \phi} \phi( F^e_* \fa^{\age{t(p^e-1)}}) \subseteq R$ (resp. $I' := \sum_{e, \phi} \phi(F^e_* (\fa \widehat{R})^{\age{t(p^e-1)}}) \subseteq \widehat{R}$), where $e$ runs through all positive integers and $\phi$ runs through all elements in $\Hom(F^e_*R(\age{(p^e-1)\D}),R)$ (resp. in $\Hom_{\widehat{R}} (F^e_* \widehat{R} (\age{(p^e-1)\widehat{\D}}), \widehat{R})$).
Then the triple $(R,\D, \fa^t)$ (resp. the triple $(\widehat{R}, \widehat{\D}, (\fa \widehat{R})^t)$) is sharply $F$-pure if and only if $I=R$ (resp. $I'=\widehat{R}$).
Since 
\[
\Hom_{\widehat{R}} (F^e_* \widehat{R} (\age{(p^e-1)\widehat{\D}}), \widehat{R}) \cong \Hom(F^e_*R(\age{(p^e-1)\D}),R) \otimes_R \widehat{R},
\]
we have $I'=I \widehat{R}$, which completes the proof.
\end{proof}

Suppose that $R$ is a ring of characteristic $p>0$, $e>0$ is a positive integer and $\fa \subseteq R$ is an ideal.
Then we denote by $\fa^{[p^e]}$ the ideal of $R$ generated by $\{ f^{p^e} \in R \mid f \in \fa \}.$
The following lemma is a variant of Fedder-type criteria.

\begin{lem}[cf. \textup{\cite{Fed}, \cite[Proposition 2.6]{HW}}]\label{Fedder}
Suppose that $(A,\m)$ is an $F$-finite regular local ring of characteristic $p>0$, $\fa \subseteq A$ is an ideal and $\D=\Div_A (f)/ (p^e-1)$ is an effective $\Q$-divisor with $f \in A$ and $e>0$.
Then, the triple $(A, \D, \fa^t)$ is sharply $F$-pure if and only if there exists an integer $n>0$ such that 
\[
f^{\frac{p^{en}-1}{p^e-1}} \fa^{\age{t(p^{en}-1)}} \not\in \m^{[p^{en}]}.
\]
\end{lem}

\begin{proof}
By the proof of \cite[Proposition 3.3]{Sch-sh}, the triple $(A, \D, \fa^t)$ is sharply $F$-pure if and only if there exists an integer $n > 0$ and $\phi \in \Hom_R(F^{en}_* A( (p^{en}-1)\D), A)$ such that $\phi( F^{en}_* \fa^{\age{t(p^{en}-1)}})=A$.
Since 
\[
(p^{en}-1)\D= \Div_A(f^{\frac{p^{en}-1}{p-1}}),
\]
the assertion follows from \cite[Lemma 1.6]{Fed}.
\end{proof}

Suppose that $X$ is an $F$-finite Noetherian normal connected scheme, $\D$ is an effective $\Q$-Weil divisor on $X$, $\fa \subseteq \sO_X$ is a coherent ideal sheaf and $t \ge 0$ is a real number.
For any point $x \in X$, we denote by $\D_x$ the flat pullback of $\D$ to $\Spec \sO_{X,x}$.

\begin{defn}
With the notation above, we say that $(X, \D, \fa^t)$ is \emph{sharply $F$-pure} if $(\sO_{X,x}, \D_x, \fa_x^t)$ is sharply $F$-pure for every point $x \in X$.
\end{defn}

\begin{rem}
Suppose that $X=\Spec R$ is an affine scheme.
Then, the above definition differs from the one given in \cite{Sch-sh}.
See \cite{Sch-ref}.
\end{rem}

\begin{defn}\label{fpt def}
With the notation above, assume that $(X, \D)$ is sharply $F$-pure.
We define the \emph{$F$-pure threshold} of $(X,\D; \fa)$ by
\[
\lfpt(X, \D; \fa) := \inf\left\{ t \ge 0 \mid (X, \D, \fa^t) \textup{ is not sharply $F$-pure} \right\} \in \R_{\ge 0} \cup \{ \infty \}.
\]
When $X= \Spec R$ is an affine scheme, we denote it by $\lfpt(R, \D; \fa)$.
\end{defn}

\begin{lem}\label{fpt min}
With the notation above, we assume that $(X, \D)$ is sharply $F$-pure.
Then, $\lfpt(X,\D;\fa) = \min\left\{ \fpt(\sO_{X,x}, \D_x; \fa_x) \mid x \in X \right\}$.
\end{lem}

\begin{proof}
We may assume that $X=\Spec R$.
For every $t \ge 0$, we consider $I_t : = \sum_{e, \phi} \phi( F^e_* \fa^{\age{t (p^e-1)}}) \subseteq R$ as in the proof of Lemma \ref{F-sing basic} (5).
Then, the set 
\[
 Z_t : = \{ x \in X \mid (\sO_{X,x}, \D_x, \fa_x^t) \textup{ is not sharply $F$-pure} \} \subseteq X\] 
is a closed set defined by the ideal $I_t$.
Since $R$ is Noetherian, there exists a real number $\epsilon>0$ such that $Z_t$ is constant for all $\fpt(X,\D; \fa)< t < \fpt(X, \D; \fa) +\epsilon$.
Take a point $x \in Z_t$ for such $t$.
Then we have $\fpt(X, \D;\fa)= \fpt(\sO_{X,x}, \D_x; \fa_x)$, which completes the proof.
\end{proof}

\begin{prop}[\textup{\cite[Theorem 5.5]{Sch-adj}}]\label{F-adj}
Suppose that $A$ is an $F$-finite regular local ring, $R=A/I$ is a normal ring and $\D_R$ is an effective $\Q$-Weil divisor on $\Spec R$.
Assume that the pair $(R,\D_R)$ is sharply $F$-pure and there exists an integer $e>0$ such that $(p^e-1)(K_X+\D_R)$ is Cartier.
Then, there exists an effective $\Q$-Weil divisor $\D_A$ on $\Spec A$ with the following properties:
\begin{enumerate}
\item $(p^e-1)\D_A$ is Cartier, and
\item Suppose that $\fa \subseteq R$ is an ideal and $\tilde{\fa} \subseteq A$ is the lift of $\fa$.
Then we have $\fpt(R, \D_R; \fa) = \fpt(A, \D_A; \tilde{\fa})$. 
\end{enumerate}
\end{prop}

\subsection{Test ideals and parameter test modules}
In this subsection, we recall the definitions and basic properties of test ideals and parameter test modules.

\begin{defn}\label{test def}
Suppose that $R$ is an $F$-finite Noetherian normal domain, $\D$ is an effective $\Q$-Weil divisor on $X=\Spec R$, $\fa, \fb \subseteq R$ are non-zero ideals and $t,s \ge 0$ are real numbers.
The \emph{test ideal} $\tau(R, \D, \fa^t \fb^s)$ (resp. the \emph{parameter test module} $\tau(\omega_X, \D, \fa^t \fb^s)$) is the unique smallest non-zero ideal $J \subseteq R$ (resp. non-zero submodule $J \subseteq \omega_X$) such that 
\[
\phi( F^e_* (\fa^{ \age{t( p^e-1)}}\fb^{ \age{s( p^e-1)}} J)) \subseteq J
\]
for every integer $e \ge 0$ and every morphism $\phi \in \Hom_R(F^e_* R (\age{(p^e-1)\D}), R)$ (resp. $\phi \in \Hom_R(F^e_* \omega_X (\age{(p^e-1)\D}), \omega_X)$).
\end{defn}

The test ideal and the parameter test module always exist (\cite[Theorem 6.3]{Sch-cent} and \cite[Lemma 4.2]{ST14}).
If $\fb=R$, then we write $\tau(R, \D, \fa^t)$ (resp. $\tau(\omega_X, \D, \fa^t)$).
If $\fa=\fb=R$, then we write $\tau(R, \D)$ (resp. $\tau(\omega_X, \D)$).
If $\fa=0$, then we define $\tau(\omega_X, \D, \fa^t)=\tau(R,\D,\fa^t) : =(0)$.

\begin{lem}\label{test basic}
With the notation above, we assume that $\D$ is $\Q$-Cartier.
Then the following hold.
\begin{enumerate}
\item If $t \le t'$ and $\D \le \D'$, then $\tau(\omega_X, \D', \fa^{t'}) \subseteq \tau(\omega_X, \D , \fa^t )$.
\item \textup{(\cite[Lemma 6.1]{ST14})} There exists a real number $\epsilon>0$ such that if $t \le t' \le t+ \epsilon$, then $\tau(\omega_X, \D, \fa^{t'} )= \tau(\omega_X, \D , \fa^t)$.
\item \textup{(\cite[Lemma 6.2]{ST14})} There exists a real number $\epsilon>0$ such that if $t-\epsilon \le t' < t$, then $\tau(\omega_X, \D, \fa^{t'} )= \tau(\omega_X , \D , \fa^{t-\epsilon})$.
\item \textup{(\cite[Lemma 4.4]{ST14})} Suppose that $\Tr_R :F_* \omega_X \to \omega_X$ is the Grothendieck trace map \textup{(\cite[Proposition 2.18]{BST})}. 
Then we have 
\[
\Tr_R ( F_* \tau( \omega_X, \D, \fa^t) ) = \tau(\omega_X, \D/p, \fa^{t/p}).
\]
\item \textup{(\cite[Theorem 4.2]{HT}, cf. \cite[Lemma 3.26]{BSTZ})} If $\fa$ is generated by $l$ elements and $l \le t$, then $\tau(\omega_X, \D, \fa^t \fb^s)= \fa \tau(\omega_X, \D, \fa^{t-1} \fb^s)$.
\item \textup{(\cite[Lemma 3.1]{Sch-note})} If $\fb= (f)$ is a non-zero principal ideal, then we have $\tau(\omega_X, \D, \fa^t \fb^s)= \tau(\omega_X, \D+ s \Div(f), \fa^t)$.
\item For an integer $r \ge 1$, we have $\tau(\omega_X, \D, \fa^{rt})= \tau(\omega_X, \D, (\fa^r)^t)$.
\end{enumerate}
\end{lem}

\begin{proof}
Take a canonical divisor $K_X$ such that $-K_X$ is effective.
Then we have $\tau( \omega_X, \D, \fa^t \fb^s)= \tau(R, \D- K_X, \fa^t \fb^s)$ (\cite[Lemma 4.2]{ST14}).
Therefore, the assertions in (5) and (6) follow from the same assertions for test ideals.
The proof of (7) is similar to the proof of (6).
\end{proof}

\begin{rem}
Suppose that $X$ is an $F$-finite Noetherian normal connected scheme which has a canonical module $\omega_X$, $\fa, \fb \subseteq \sO_X$ are coherent ideals, and $t, s \ge 0$ are real numbers.
Since parameter test modules are compatible with localization (\cite[Proposition 3.1]{HT}), we can define the parameter test module $\tau(\omega_X, \D, \fa^t \fb^s) \subseteq \omega_X$.
\end{rem}

\begin{lem}\label{tau- vs Fpure}
Let $(A,\D, \fa^t)$ be a triple such that $A$ is a regular local ring.
\begin{enumerate}
\item If $(A, \D, \fa^t)$ is sharply $F$-pure, then for any rational numbers $0< \epsilon, \epsilon'<1$, the triple $(A, (1-\epsilon) \D, \fa^{t(1-\epsilon')})$ is strongly $F$-regular.
\item If $(p^e-1)\D$ is Cartier for an integer $e>0$ and $(A, (1-\epsilon) \D, \fa^t)$ is strongly $F$-regular for every $0<\epsilon<1$, then the triple $(A,\D, \fa^{t(1-\epsilon')})$ is sharply $F$-pure for every $0<\epsilon'<1$.
\end{enumerate} 

\end{lem}

\begin{proof}
For (1), we assume that the triple $(A, \D, \fa^t)$ is sharply $F$-pure.
Since $A$ is strongly $F$-regular (\cite{HH}), it follows from Lemma \ref{F-sing basic} (2) and (3) that $(A, (1-\epsilon) \D)$ is strongly $F$-regular for every $0<\epsilon < 1$.
Then applying Lemma \ref{F-sing basic} (2) and (3) again, we see that $(A, (1-\epsilon)\D, \fa^{t(1-\epsilon')})$ is strongly $F$-regular for every $0<\epsilon, \epsilon' <1$.

For (2), set $q := p^e$ and suppose that $\D= \Div(f)/(q-1)$ for some non-zero element $f \in A$.
Take an integer $l >t$ such that $\fa$ is generated by at most $l$ elements and set $a_n : = (l-t)/ (q^n-1)$ for every integer $n \ge 0$.
Since for any triple, it is strongly $F$-regular if and only if the test ideal is trivial (\cite[Corollary 2.10]{Tak}, see also \cite[Corollary 4.6]{Sch-cent}), we have $\tau(A, ((q^n-1)/q^n) \D /, \fa^t) = A$ for every integer $n \ge 0$.

Since $A$ is regular local, we may identify test ideals on $A$ with parameter test modules.
Set $\phi : = \Tr^e_A ( F^e_*(f \cdot - )) \in \Hom_A(F^e_* A( (q-1) \D, A))$.
Then it follows from Lemma \ref{test basic} (4), (5) and (6) that
\begin{eqnarray*}
A &=& \tau(A, ((q^n-1)/q^n) \D , \fa^t)\\
 &=& \Tr^{en}_A ( F^{en}_* \tau(A, f^{(q^n-1)/(q-1)} \fa^{t q^n})) \\
&=& \phi^n( F^{en}_* \tau(A, \fa^{t q^n})) \\
&\subseteq & \phi^n( F^{en}_* \fa^{\age{t q^n -l}} )\\
&\subseteq & \phi^n( F^{en}_* \fa^{\age{(t-a_n) (q^n-1)}} ),
\end{eqnarray*}
which proves that $(A, \D, \fa^{t-a_n})$ is sharply $F$-pure for every integer $n \ge 0$.
Since $\lim_{n \to \infty} a_n=0$, the triple $(A, \D, \fa^{t(1- \epsilon')})$ is sharply $F$-pure for every $0<\epsilon' <1$.
\end{proof}

\subsection{Ultraproduct}
In this subsection, we define the catapower of a Noetherian local ring and recall some properties.

\begin{defn}
Let $\U$ be a collection of subsets of $\N$.
$\U$ is called an \emph{ultrafilter} if the following properties hold:
\begin{enumerate}
\item $\emptyset \not\in \U$.
\item For every subsets $A, B \subseteq \N$, if $A \in \U$ and $A \subseteq B$, then $B \in \U$.
\item For every subsets $A, B \subseteq \N$, if $A, B \in \U$, then $A \cap B \in \U$.
\item For every subset $A \subseteq \N$, if $A \not\in \U$, then $\N \setminus A \in \U$.
\end{enumerate}

An ultrafilter $\U$ is called \emph{non-principal} if the following holds:
\begin{enumerate}
\setcounter{enumi}{4}
\item If $A$ is a finite subset of $\N$, then $A \not\in \U$.
\end{enumerate}
\end{defn}

By Zorn's Lemma, there exists a non-principal ultrafilter.
From now on, we fix a non-principal ultrafilter $\U$.

\begin{defn}
Let $T$ be a set.
We define the equivalence relation $\sim$ on the set $T^{\N}$ by
\[
(a_m)_m \sim (b_m)_m \textup{ if and only if } 
\left\{ m \in \N \mid a_m=b_m \right\} \in \U.
\]
We define the \emph{ultrapower} of $T$ as
\[
\ultra{T} : = T^{\N}/ \sim.
\]
\end{defn}

The class of $( a_m )_m \in T^{\N}$ is denoted by $\ulim_m a_m$.
If $T$ is a ring (resp. local ring, field), then so is $\ultra{T}$.
Moreover, if $T$ is an $F$-finite field of characteristic $p>0$, then so is $\ultra{T}$. (see \cite[Proposition 2.14]{Sat}).

\begin{defn}[\textup{\cite{Scho}}]
Suppose that $(R, \m)$ is a Noetherian local ring and $( \ultra{R}, \ultra{\m})$ is the ultrapower.
We define the \emph{catapower} $\cata{R}$ as the quotient ring
\[
\cata{R} : = \ultra{R}/ (\cap_{n} (\ultra{\m})^n).
\]
\end{defn}

\begin{prop}[\textup{\cite[Theorem 8.1.19]{Scho}}]
Suppose that $(R, \m, k)$ is a Noetherian local ring of equicharacteristic and $\widehat{R}$ is the $\m$-adic completion of $R$.
We fix a coefficient field $k \subseteq \widehat{R}$.
Then we have 
\[
\cata{R} \cong \widehat{R} \ \widehat{\otimes}_k (\ultra{k}).
\]
In particular, if $(R,\m)$ is an $F$-finite regular local ring, then so is $\cata{R}$.
\end{prop}

Suppose that $(R, \m)$ is a Noetherian local ring, $\cata{R}$ is the catapower and $a_m \in R$ for every $m$.
We denote by $\catae{a_m} \in \cata{R}$ the image of $(a_m)_m \in R^{\N}$ by the natural projection $R^{\N} \to \cata{R}$.
Let $\fa_m \subseteq R$ be an ideal for every $m \in \N$.
We denote by $[ \fa_m]_m \subseteq \cata{R}$ the image of the ideal $ \prod_m \fa_m \subseteq R^{\N}$ by the projection $R^{\N} \to \cata{R}$.

\begin{propdef}[\textup{\cite[Theorem 5.6.1]{Gol}}]
Let $\{a_m \}_{m \in \N}$ be a sequence of real numbers such that there exist real numbers $M_1, M_2$ which satisfies $M_1<a_m<M_2$ for every $m \in \N$.
Then there exists an unique real number $w \in \R$ such that for every real number $\epsilon >0$, we have
\[
\{ m \in \N \mid |w-a_m| <\epsilon \} \in \U.
\]
We denote this number $w$ by $\sh( \ulim_m a_m)$ and call it the \emph{shadow} of $\ulim_m a_m \in \ultra{\R}$.
\end{propdef}

\section{A variant of parameter test modules}
In this section, we define a variant of parameter test modules and prove the rationality of $F$-jumping numbers.

\begin{prop}\label{discrete1}
Let $(X=\Spec R, \D, \fa^t)$ be a triple such that $\D= s D$ for some Cartier divisor $D$ and $t=s=1/{(p^e-1)}$ for some integer $e>0$.
Then $\tau(\omega_X, (s-\epsilon) D, \fa^t)$ is constant for all sufficiently small rational numbers $0 < \epsilon \ll 1$.
\end{prop}

\begin{proof}
The proof is essentially the same as that of \cite[Lemma 6.2]{ST14}.
We may assume that $\fa \neq 0$.
Set $q=p^e$.
For every integer $l \ge 0$, we define the $l$-th truncation of $s$ in the base $q$ by 
\[
\qadic{s}{l} : = \frac{q^l-1}{q^l (q-1)} \in \Q.
\]
Since the sequence $\{ \qadic{s}{l} \}_{l \in \N }$ is a strictly ascending chain which converges to $s$, it is enough to prove that $\tau(\omega_X, \qadic{s}{l} \cdot D, \fa^t)$ is constant for all sufficiently large $l$.

Take the normalized blowup $\pi : Y \to X$ along $\fa$.
Let $G$ be the Cartier divisor on $Y$ such that $\sO_Y(-G) = \fa \cdot \sO_Y$.
Take the Grothendieck trace maps $\Tr_{\pi}: \pi_* \omega_Y \to \omega_X$, $\Tr_X: F_* \omega_X \to \omega_X$ and $\Tr_Y: F_* \omega_Y \to \omega_Y$ (\cite[Proposition 2.18]{BST}).
As in \cite[p.4]{BST}, we have $\Tr_X \circ F_*( \Tr_{\pi})= \Tr_{\pi} \circ \pi_*(\Tr_Y)$ and $\Tr_{\pi}$ is injective. 
In particular, we may consider $\pi_* \omega_Y$ as a submodule of $\omega_X$.

By \cite[Theorem 5.1]{ST14}, for every integer $l \ge 0$, there exists an integer $m_l$ such that 
\begin{equation}\label{ST}
\tau(\omega_X, \qadic{s}{l} \cdot D, \fa^t) = \Tr_X^{e m}(F^{em}_* \pi_*(\tau( \omega_Y, q^m(\qadic{s}{l} \cdot \pi^*D + t G))))
\end{equation}
for all $m \ge m_l$.

By Lemma \ref{test basic} (3) and (6), there exists $l_0$ such that $\tau( \omega_Y, \qadic{s}{l}{} \cdot \pi^* D + t G)$ is constant for all $l \ge l_0$.
For every integer $l \ge 0$, it follows from Lemma \ref{test basic} (4) that the morphism
\[
\beta_l : = \Tr_Y^e : F^e_*( \tau( \omega_Y, q( \qadic{s}{l} \cdot \pi^* D + tG)))) \to \tau(\omega_Y,  \qadic{s}{l} \cdot \pi^* D + tG)
\]
is surjective.
We denote the kernel by $\NN_l$.
Since $\NN_l$ is constant for all $l \ge l_0$ and $-G$ is $\pi$-ample, there exists an integer $m'$ such that 
\[
R^1 \pi_* (\NN_l \otimes_{\sO_Y} \sO_Y(-MG))=0
\]
for all integers $l \ge 0$ and $M \ge (q^{m'}-1)/(q-1)$.

Take integers $m, n \ge 1$ and consider the surjection
\[
\gamma_{n,m} : = \Tr_Y^e : F^e_* (\tau(\omega_X, q^m( \qadic{s}{n} \pi^* D + t G))) \to \tau(\omega_X, q^{m-1}( \qadic{s}{n} \pi^* D + t G)).
\]
By Lemma \ref{test basic} (5) and (6), $\gamma_{n,m}$ coincides with $\beta_{n-m} \otimes \sO_Y( -  (q^m-1)/(q-1) \cdot (\pi^* D +G))$ if $m<n$ and with $\beta_0 \otimes \sO_Y(- q^m \qadic{s}{n}\pi^* D  -  (q^m-1)/(q-1) \cdot G)$ if $m \ge n$.
Therefore, $\pi_* \gamma_{n,m}$ is surjective if $m \ge m'$.

Combining with the equation (\ref{ST}), we have 
\[
\tau(\omega_X, \qadic{s}{l} \cdot D, \fa^t) = \Tr_X^{e m'}(F^{em'}_* \pi_*(\tau( \omega_Y, q^{m'}(\qadic{s}{l} \cdot \pi^*D + t G))))
\]
for every $l$.
By the definition of $l_0$, the right hand side is constant for all $l \ge l_0 + m'$.
\end{proof}

\begin{cor}\label{discrete2}
Let $(X= \Spec R, \D, \fa^t)$ be a triple such that $t \in \Q$ and $\D$ is $\Q$-Cartier.
Then $\tau(\omega_X, (1 - \epsilon) \D, \fa^t)$ is constant for all $0 < \epsilon \ll 1$.
\end{cor}

\begin{proof}
By Lemma \ref{test basic} (4) and (7), we may assume that there exists an integer $e>0$ such that $(p^e-1)\D$ is Cartier and $t=1 / (p^e-1)$.
Then the assertion follows from Proposition \ref{discrete1}. 
\end{proof}

We define the new variant of the parameter test module as the left limit of the map $s \mapsto \tau(\omega_X, s \D, \fa^t)$ at $s=1$.

\begin{defn}\label{ntau def}
Let $(X= \Spec R, \D, \fa^t)$ be a triple such that $t \in \Q$ and $\D$ is $\Q$-Cartier.
Then we define the submodule $\ntau{\D}{\fa^t} \subseteq \omega_X$ by $\tau(\omega_X, (1- \epsilon) \D, \fa^t)$ for sufficiently small $0<\epsilon \ll 1$.
\end{defn}

\begin{lem}\label{tau- basic}
Let $(X=\Spec R, \D, \fa^t)$ be a triple such that $t \in \Q$ and $\D$ is $\Q$-Cartier.
Then the following hold.
\begin{enumerate}
\item For any rational number $t<t'$, we have $\ntau{\D}{\fa^{t'}} \subseteq \ntau{\D} {\fa^t}$.
\item For any real number $s \ge 0$, there exists $0<\epsilon$ such that $\ntau{\D}{\fa^{s'}}$ is constant  for every rational number $s< s' < s+\epsilon$.
\item For any rational number $s > 0$, there exists $0<\epsilon$ such that $\ntau{\D}{\fa^{s'}}$ is constant  for every rational number $s- \epsilon < s' < s$.
\item If $\fa$ is generated by $l$ elements and $t \ge l$, then we have $\ntau{\D}{ \fa^t}=\fa \ntau{\D}{\fa^{t-1}}$.
\item $\Tr_X(F_* (\ntau{\D}{\fa^t}))= \ntau{(\D/p)}{ \fa^{t/p}}$.
\item If $r \D$ is Cartier, then $\ntau{(r+1) \D}{\fa^t}= \ntau{\D}{\fa^t} \otimes \sO_X(-r \D)$.
\end{enumerate}

\end{lem}

\begin{proof}
(1), (4), (5) and (6) follow from Lemma \ref{test basic}.
(2) follows from (1) and the ascending chain condition for the set of ideals in $R$.

For (3), we take a positive integer $r$ such that $r s$ is integer and $r \D$ is Cartier.
By Lemma \ref{test basic} (3), there exists $\delta>0$ such that $\tau(\omega_X, (\fa^{rs} \sO_X (-r \D))^{(1-\epsilon)/r})$ is constant for all rational numbers $0 < \epsilon < \delta$.
We denote this module by $M$.

It follows from Lemma \ref{test basic} (6) and (7) that for every rational number $0<\epsilon<\delta$, we have
\begin{eqnarray*}
\tau( \omega_X, (1-\epsilon) \D, \fa^{s (1-\epsilon)}) &=& \tau(\omega_X, \fa^{ s (1-\epsilon)} \sO_X (-r \D)^{(1- \epsilon)/r}) \\
&=& \tau(\omega_X, (\fa^{r s} \sO_X (-r \D))^{(1-\epsilon)/r})\\
&=& M.
\end{eqnarray*}
By Lemma \ref{test basic} (1), $\tau( \omega_X, (1-\epsilon) \D, \fa^{s (1-\epsilon')})=M$ for every $0< \epsilon , \epsilon' < \delta$.
Therefore, we have $\ntau{\D}{\fa^{s (1-\epsilon)}}=M$ for every rational number $0<\epsilon < \delta$.
\end{proof}

\begin{defn}
Let $(X= \Spec R, \D, \fa^t)$ be a triple such that $t$ is not a rational number and $\D$ is $\Q$-Cartier.
By Lemma \ref{tau- basic} (2), there exists $\epsilon >0$ such that the submodule $\ntau{ \D}{ \fa^{s}} \subseteq \omega_X$ is constant for every rational number $t<s<t+\epsilon$.
We denote this submodule of $\omega_X$ by $ \ntau{ \D}{ \fa^t}$.
\end{defn}

We note that even if $t$, $t'$, $s$, and $s'$ are not rational, the same assertions as in Lemma \ref{tau- basic} (1), (2), (4), (5) and (6) hold.

\begin{defn}\label{fjn def}
Let $(X= \Spec R, \D, \fa)$ be a triple such that $\D$ is $\Q$-Cartier.
A real number $t \ge 0$ is called an \emph{$F$-jumping number} of $\ntri{\D}{\fa}$ if one of the following hold:
\begin{enumerate}
\item for every $\epsilon>0$, we have $\ntau{\D}{ \fa^t} \subsetneq \ntau{\D}{\fa^{t-\epsilon}}$, or
\item for every $\epsilon>0$, we have $\ntau{\D}{\fa^t} \supsetneq \ntau{\D}{\fa^{t+\epsilon}}$.
\end{enumerate}
\end{defn}

\begin{lem}\label{rationality lemma}
Let $q \ge 2$ and $l \ge 1$ be integers and $B \subseteq \R_{\ge 0}$ a subset.
$B$ is a discrete set of rational numbers if the following four properties hold:
\begin{enumerate}
\item For any $x \in B$, $q x \in B$.
\item For any $x \in B$, if $x>l$, then $x-1 \in B$.
\item For any real number $t \in \R_{\ge 0}$, there exists $\epsilon>0$ such that $B \cap (t, t+\epsilon) = \emptyset$.
\item For any rational number $t \in \Q_{> 0}$, there exists $\epsilon>0$ such that $B \cap (t-\epsilon, t) = \emptyset$.
\end{enumerate}
\end{lem}

\begin{proof}
Let $D$ be the set of all accumulation points of $B$.
By \cite[Proposition 5.5]{BSTZ}, we have $D= \emptyset$.
This proves that $B$ is a discrete set.
If $B$ contains a non-rational number, then by the assumptions (1) and (2), we have infinitely many elements in $B \cap [l-1, l]$, which contradicts to the discreteness of $B$.
\end{proof}

\begin{cor}\label{F-jump rat}
Let $(X= \Spec R , \D , \fa)$ is a triple such that $\D$ is $\Q$-Cartier.
Then the set of all $F$-jumping numbers of $\ntri{\D}{\fa}$ is a discrete set of rational numbers.
\end{cor}

\begin{proof}
It follows from Lemma \ref{tau- basic} (5) that if $t$ is an $F$-jumping number of $\ntri{ \D} {\fa}$, then $p t$ is an $F$-jumping number of $\ntri{(p \D) }{\fa}$.
Therefore, we may assume that there exists an integer $e>0$ such that $(p^e-1)\D$ is Cartier. 

Let $l$ be the number of minimal generators of $\fa$ and $B$ be the set of all $F$-jumping numbers of $\ntri{\D}{\fa}$.
Then it follows from Lemma \ref{tau- basic} that $B$, $q=p^e$ and $l$ satisfy the assumptions in Lemma \ref{rationality lemma}.
\end{proof}

\section{Proof of Main Theorem}
In this section, applying Corollary \ref{F-jump rat}, we prove the rationality of $F$-pure thresholds (Corollary \ref{rat}).
We also prove that the shadow of $F$-pure thresholds coincides with the $F$-pure threshold on the catapower (Theorem \ref{fpt sh 2}).
By combining them, we give the proof of the main theorem (Theorem \ref{main}).

\begin{prop}\label{fpt fjn}
Suppose that $(X=\Spec A, \D)$ is a sharply $F$-pure pair such that $A$ is regular and $(p^e-1)\D$ is Cartier for some $e>0$, and $\fa \subseteq A$ is a non-zero proper ideal.
Then the $F$-pure threshold $\fpt(A,\D; \fa)$ coincides with the first jumping number of $\ntri{\D}{\fa}$.
In particular, it is a rational number.
\end{prop}

\begin{proof}
It is enough to show the equation
\begin{equation}\label{4.2}
\fpt(A, \D; \fa)= \sup \left\{s \ge 0 \mid \ntau{\D}{ \fa^s} = \omega_X \right\}.
\end{equation}
Set $t : = \fpt (A, \D; \fa)$.
Since $A$ is regular local, we may identify $\omega_X$ with $A$.
By Lemma \ref{tau- vs Fpure} (1), we have $\ntau{\D}{\fa^{t(1-\epsilon)}}=\omega_X$ for every $0 < \epsilon<1$.

On the other hand, take any rational number $s$ such that $\ntau{\D}{\fa^s} = \omega_X$.
It follows from Lemma \ref{tau- vs Fpure} (2) that $(A, \D, \fa^{s(1-\epsilon)})$ is sharply $F$-pure for every $0<\epsilon <1$, which proves the equation (\ref{4.2}).
\end{proof}

\begin{cor}[Theorem \ref{intro rat}]\label{rat}
Suppose that $(R,\D)$ is a sharply $F$-pure pair such that $(p^e-1)(K_R+\D)$ is Cartier for some integer $e>0$ and $\fa \subseteq R$ is an ideal.
Then the $F$-pure threshold $\fpt(R,\D; \fa)$ is a rational number.
\end{cor}

\begin{proof}
By Lemma \ref{F-sing basic} (5), we may assume that $R$ is a complete local ring.
By Proposition \ref{F-adj}, we may assume that $R$ is a regular local ring.
Hence, the assertion follows from Proposition \ref{fpt fjn}.
\end{proof}

\begin{lem}\label{fpt concent}
Suppose that $A$ is an $F$-finite regular local ring, $f \in A$ is a non-zero element, $\fa \subseteq A$ is an ideal, $e>0$ is an integer and $t=u/v > 0$ is a rational number with integers $u,v>0$.
Set $\fb : = f^v \cdot \fa ^{(p^e-1)u} \subseteq A$ and $\D : = \Div_A (f)/ (p^e-1)$.
Assume that $(A, \D)$ is sharply $F$-pure.
Then $t \le \fpt(A, \D ; \fa)$ if and only if $1/(v(p^e-1)) \le \fpt(A; \fb)$.

\end{lem}

\begin{proof}
We may assume that $\fa \neq (0)$.
First, we assume that $t \le \fpt(A, \D; \fa)$.
By Lemma \ref{tau- vs Fpure} (1), the triple $(A, (1-\epsilon) \D, \fa^{(1-\epsilon) t})$ is strongly $F$-regular for every $0 < \epsilon <1$.
It follows from Lemma \ref{F-sing basic} (4) that the triple $(A, \fb^{(1-\epsilon)/(v(p^e-1))})$ is strongly $F$-regular, which proves the inequality $1/(v(p^e-1)) \le \fpt(A; \fb)$.

On the other hand, we assume that $1/(v(p^e-1)) \le \fpt(A; \fb)$.
By Lemma \ref{F-sing basic} (3) and (4), the triple $(A, (1-\epsilon) \D, \fa^{(1-\epsilon)t})$ is strongly $F$-regular for every $0< \epsilon<1$.
It follows from \ref{F-sing basic} (2) that the triple $(A, (1-\epsilon)\D, \fa^{(1-\epsilon')t})$ is strongly $F$-regular for every $0< \epsilon, \epsilon'<1$.
By Lemma \ref{tau- vs Fpure} (2), we have $t \le \fpt(A, \D; \fa)$.
\end{proof}

\begin{prop}\label{fpt sh 1}
Suppose that $A$ is an $F$-finite regular local ring, $e>0$ is an integer, $\D_m = \Div_A (f_m) /(p^e-1)$ is an effective $\Q$-divisor on $\Spec A$ for every $m \in \N$ and $\fa_m \subseteq A$ is a proper ideal for every $m \in \N$.
Fix a non-principal ultrafilter $\U$.
Let $\cata{A}$ be the catapower of $A$ and $\fa_\infty : = \catae{\fa_m} \subseteq \cata{A}$.
Assume that $(A, \D_m)$ is sharply $F$-pure for every integer $m$.
Then the following hold.
\begin{enumerate}
\item $f_\infty := \catae{f_m} \in \cata{A}$ is a non-zero element.
\item Set $\D_\infty := \Div_{\cata{A}}(f_\infty)/(p^e-1)$. 
Then, $(\cata{A}, \D_\infty)$ is sharply $F$-pure.
\item For every rational number $t > 0$, we have $t \le \fpt(\cata{A}, \D_\infty ; \fa_\infty)$ if and only if $\{ m \in \N \mid t \le \fpt(A, \D_m ; \fa_m)\} \in \U$.
\end{enumerate}

\end{prop}

\begin{proof}
By Lemma \ref{Fedder}, we have $f_m \not\in \m^{[p^e]}$ for every $m$.
It follows from \cite[Lemma 2.19]{Sat} that $f_\infty \not\in \m^{[p^e]}$, which proves (1) and (2).
For (3), take integers $u,v>0$ such that $t=u/v$ and set $\fb_m := f_m^v \cdot \fa_m^{u(p^e-1)}$ for every $m \in \N \cup \{ \infty \}$.
It follows from Lemma \ref{fpt concent} that $\{ m \in \N \mid t \le \fpt(A, \D_m; \fa_m) \} \in \U$ if and only if $\{ m \in \N \mid 1/(v(p^e-1)) \le \fpt(A; \fb_m) \} \in \U$.
We first assume that $\{ m \in \N \mid 1/(v(p^e-1)) \le \fpt(A; \fb_m) \} \in \U$.
Since we have $\sh( \ulim_m \fpt(A; \fb_m)) =\fpt(\cata{A}; \fb_\infty)$ (\cite[Theorem 4.7]{Sat}), we have $1/(v(p^e-1)) \le \fpt(\cata{A}; \fb_\infty)$.
Applying Lemma \ref{fpt concent} again, we have $t \le \fpt(\cata{A}, \D_\infty; \fa_\infty)$.

For the converse implication, we assume that $\{ m \in \N \mid 1/(v(p^e-1)) \le \fpt(A; \fb_m) \} \not\in \U$.
In this case, we have $\{ m \in \N \mid 1/(v(p^e-1)) > \fpt(A; \fb_m) \} \in \U$ and hence we have $1/(v(p^e-1)) \ge \fpt(\cata{A}; \fb_\infty)$.
If $1/(v(p^e-1)) = \fpt(\cata{A}; \fb_\infty)=\sh( \ulim_m \fpt(A; \fb_m))$, then by replacing by a subsequence, we may assume that the sequence $\{ \fpt(A; \fb_m)\}_m$ is a strictly ascending chain, which is contradiction to \cite[Main Theorem]{Sat}.
Therefore, we have $1/(v(p^e-1)) > \fpt(\cata{A}; \fb_\infty)$, which proves $t> \fpt(\cata{A}, \D_\infty; \fa_\infty)$.
\end{proof}

\begin{thm}\label{fpt sh 2}
With the notation above, we have 
\[
\sh( \ulim_m \fpt (A, \D_m ; \fa_m) ) = \fpt (\cata{A}, \D_{\infty}, \fa_\infty) \in \Q.
\]
In particular, if the limit $\lim_{m \to \infty} \fpt (A, \D_m; \fa_m)$ exists, then we have
\[
\lim_{m\to \infty} \fpt (A, \D_m; \fa_m) = \fpt (\cata{A}, \D_\infty, \fa_\infty).
\]
\end{thm}

\begin{proof}
We first note that the shadow always exists because we have $\fpt(A, \D_m; \fa_m) \le \fpt(A; \m) = \dim A$ for all $m$.
For any rational number $t > 0$, it follows from Proposition \ref{fpt sh 1} that $t \le \sh(\ulim_m \fpt(A, \D_m; \fa_m))$ if and only if $t \le \fpt( \cata{A}, \D_\infty; \fa_\infty)$, which completes the proof.
\end{proof}

\begin{cor}\label{ACC for RLR}
Suppose that $e>0$ is an integer and $(A, \m)$ is an $F$-finite regular local ring of characteristic $p>0$.
Then the set 
\[
\FPT(A, e) := \left\{ \fpt(A, \D ; \fa) \mid (A, \D) \textup{ is sharply $F$-pure}, (p^e-1)\D \textup{ is Cartier, }  \fa \subsetneq A \right\}
\]
satisfies the ascending chain condition.
\end{cor}

\begin{proof}
We assume the contrary.
Then there exist sequences $\{\D_m \}_m$ and $\{\fa_m\}$ such that $\{ \fpt(A, \D_m ; \fa_m) \}_{m \in \N}$ is a strictly ascending chain.
Set $t : = \lim_m \fpt( A, \D_m; \fa_m)$.
By Corollary \ref{rat} and Corollary \ref{fpt sh 2}, we have $t =\fpt(\cata{A}, \D_\infty; \fa_\infty) \in \Q$.

Since $t$ is rational and $\fpt(A, \D_m ; \fa_m)<t$ for all $m$, it follows from Proposition \ref{fpt sh 1} that $ \fpt(A, \D_\infty ; \fa_\infty) <t $, which is contradiction.
\end{proof}

For a Noetherian local ring $(R,\m)$, we denote by $\emb(R)$ the embedding dimension of $R$.

\begin{thm}[Main Theorem]\label{main}
Fix positive integers $e$ and $N$.
Suppose that $T$ is any set such that every element of $T$ is an $F$-finite Noetherian normal local ring $(R,\m)$ with $\emb(R) \le N$.
Let $\FPT(T,e)$ be the set of all $F$-pure thresholds $\fpt(R, \D ;\fa)$ such that
\begin{itemize}
\item $R$ is an element of $T$,
\item $\fa$ is a proper ideal of $R$, and
\item $\D$ is an effective $\Q$-Weil divisor on $X=\Spec R$ such that $(R,\D)$ is sharply $F$-pure and $(p^e-1)(K_X+\D)$ is Cartier.
\end{itemize}
Then the set $\FPT(T, e)$ satisfies the ascending chain condition.
\end{thm}

\begin{proof}
Take an $F$-finite field $k$ such that for every $(R,\m) \in T$, there exists a field extension $R/\m \subseteq k$.
Set $A : = k[[x_1, \dots, x_N]]$.
Then it follows from Lemma \ref{F-sing basic} (6), Proposition \ref{F-adj} and Lemma \ref{Fedder} that we have the inclusion $\FPT(T, e) \subseteq \FPT(A, e)$, which proves that the set $\FPT(T, e)$ satisfies the ascending chain condition.
\end{proof}

\begin{cor}\label{var ACC}
Suppose that $X$ is a normal variety over an $F$-finite field.
Fix an integer $e>0$.
Let $\FPT(X, e)$ be the set of all $\lfpt(X, \D;\fa)$ such that
\begin{itemize}
\item $\fa$ is a proper coherent ideal sheaf on $X$ and 
\item $\D$ is an effective $\Q$-Weil divisor on $X$ such that $(X,\D)$ is sharply $F$-pure and $(p^e-1)(K_X+\D)$ is Cartier.
\end{itemize}
The set $\FPT(X,e)$ satisfies the ascending chain condition.
\end{cor}

\begin{proof}
Set $T : = \{ \sO_{X,x} \mid x \in X \}$.
It follows from Lemma \ref{fpt min} that $\FPT(X, e) \subseteq \FPT(T , e)$, which completes the proof.
\end{proof}

\begin{lem}[cf. \textup{\cite[Proposition 6.3]{dFEM}}]\label{lci emb}
Let $(R,\m)$ be an $F$-finite Noetherian normal local ring of dimension $d$.
If $R$ is a complete intersection and sharply $F$-pure, then $\emb(R) \le 2 d$.
\end{lem}

\begin{proof}
Set $N: = \emb(R)$ and $c: = N-d$.
There exists an $F$-finite regular local ring $A$ and a regular sequence $f_1, \dots, f_c \in A$ with $f_i \in \m^2 $ such that $R \cong A/(f_1, \dots , f_c)$.
By \cite[Proposition 2.6]{HW}, we have $(f_1 \cdots f_c)^{p-1} \not\in \m^{[p]}$.

Since $f_i \in \m^2$ for every $i$, we have $(f_1 \cdots f_c)^{p-1} \in \m^{2 c(p-1)}$.
It follows from the inclusion $\m^{N (p-1) +1} \subseteq \m^{[p]}$ that we have $2c \le N$, which proves $N \le 2d$.
\end{proof}

\begin{cor}[Theorem \ref{intro lci}]\label{lci ACC}
Let $n \ge 1$ be an integer.
Suppose that $T$ is any set such that every element of $T$ is an $n$-dimensional Noetherian normal connected l.c.i. scheme which is sharply $F$-pure.
Then, the set 
\[
 \{ \lfpt(X; \fa) \mid X \in T, \fa \subsetneq \sO_X  \}
\]
satisfies the ascending chain condition.

\end{cor}

\begin{proof}
It follows from Lemma \ref{lci emb} that $\emb(\sO_{X,x}) \le 2 n$ for every $X \in T$ and every $x \in X$.
Since every $X \in T$ is Gorenstein, we apply the main theorem.
\end{proof}







\end{document}